\documentclass[oneside,english,draft]{article}
\usepackage[utf8]{inputenc}
\usepackage[usenames,dvips]{color}
\usepackage[sort&compress]{natbib}

\usepackage[left=2.5cm,top=2.5cm,right=2.5cm,,bottom=3.5cm,nohead]{geometry}

\usepackage[T1]{fontenc}
\usepackage{amssymb,amsthm,amsmath}
\usepackage{babel}

\newcommand{\tv}{\tilde{v}_T}
\newcommand{\eexp}[1]{\exp \left( #1 \right)}

\usepackage[usenames]{color}
\definecolor{Red}{rgb}{0.9,0.1,0.3}

\newcommand{\ev}[1]{\mathbb{E}{#1}}

\newcommand{\rbr}[1]{\left( #1 \right)}
\newcommand{\sbr}[1]{\left[ #1 \right]}
\newcommand{\cbr}[1]{\left\{ #1 \right\}}
\newcommand{\ddp}[2]{\left\langle #1, #2 \right\rangle}
\newcommand{\intr}{\int_{\mathbb{R}^d}}
\newcommand{\inti}{\int_{0}^{+\infty}}
\newcommand{\intc}[1]{\int_{0}^{#1}}
\newcommand{\Rd}{\mathbb{R}^d}

\newcommand{\T}[1]{\mathcal{T}_{#1}}

\definecolor{czerwony}{rgb}{1,0.3,0.3}

\newcommand{\dd}[1]{\textnormal{d}#1}

\newcommand{\SP}{\mathcal{S}'(\Rd)}
\newcommand{\SD}{\mathcal{S}(\Rd)}

\theoremstyle{plain} 
\newtheorem{thm}{Theorem}[section]

\newtheorem{lem}{Lemma}[section]
\newtheorem{prop}{Proposition}[section]

\theoremstyle{remark}\newtheorem{rem}{Remark}[section]
\newtheorem*{acknowledgement*}{Acknowledgement}

\title{Occupation times of subcritical branching immigration system with Markov
motions}
\author{\textsc{Piotr Miłoś}}
\date{}
\begin{document}
\maketitle
\textit{Institute of Mathematics, Polish Academy of Sciences. Śniadeckich 8, 00-956 Warsaw}

\textit{ E-mail: pmilos@mimuw.edu.pl, Tel: +48600973193}
\abstract{We consider a branching system consisting of particles moving  according to a Markov family in $\Rd$ and undergoing subcritical branching with a constant rate $V>0$. New particles immigrate to the system according to homogeneous space-time Poisson random field. The process of the fluctuations of the rescaled occupation time is studied with very mild assumptions on the Markov family. In this general setting a functional central limit theorem is proved. The subcriticality of the branching law is crucial for the limit behaviour and in a sense overwhelms the properties of the particles' motion. It is for this reason that the limit is the same for all dimensions and can be obtained for a wide class of Markov processes. Another consequence is the form of the limit -  $\SP$-valued Wiener process with a simple temporal structure and a complicated spatial one. This behaviour contrasts sharply with the case of critical branching systems (cf. more detailed description in Introduction).}\\



\section{Introduction}
In this paper, we consider the following branching particle system with immigration. Particles evolve independently in $\Rd$ according to a time-homogeneous Markov family $(\eta_t,\mathbb{P}_x)_{t\geq0,x\in \Rd}$. The lifetime of a particle distributed exponentially with a parameter $V>0$. When dying the particle splits according to a binary branching law, determined by the generating function
\begin{equation}
 F(s) = qs^2 + (1-q),\qquad q<1/2. \label{eq:generating}
\end{equation}
This branching law is \textit{subcritical} (i.e. number of particles spawning from one is strictly less then $1$). Each of the new-born particles undertakes movement according to the Markov family $\eta$ independently of the others, branches, and so on. New particles \textit{immigrate} randomly to the system according to a homogeneous Poisson random field in $\mathbb{R}_+\times \Rd$ (i.e. time and space) with the intensity measure $H \lambda_{d+1}$, $H>0$ (where $\lambda_{d+1}$ denotes $d+1$ dimensional Lebesgue measure). Because of immigration the initial distribution of particles does not affect the system in long term. For the sake of  simplicity, we choose it to be a Poisson random field in $\Rd$ with intensity $L \lambda_d$, $L>0$. All random objects, the evolution of particles, the immigration and the initial distribution are (conditionally) independent.\\
The evolution of the system is described by (and in fact can be regarded to be identical with) the empirical (measure-valued) process $(N_t)_{t\geq 0}$, where $N_t(A)$ denotes the number of particles in the set $A\subset\Rd$ at time $t$. We define the fluctuations of the rescaled occupation time process by
\begin{equation}
X_T(t) = \frac{1}{F_T} \intc{Tt} \left( N_s - \ev N_s \right) \dd{s}, \: t \geq 0, \label{def:occupation_process}
\end{equation}
where $T$ is a scaling parameter which accelerates time ($T\rightarrow +\infty$) and $F_T$ is a proper deterministic norming. $X_T$ is a signed-measure-valued process, but it is convenient to regard it as a process in the space of tempered distributions $S'(\Rd)$. The objectives are to find \textit{suitable $F_T$}, such that $X_T$ converges in law as $T\rightarrow +\infty$ to a \textit{non-trivial limit} and to identify this limit.\\
We will discuss some of the related work on the fluctuations of the rescaled occupation time first as it will make it easier to understand our result. The series of papers \cite{Bojdecki:2006ab,Bojdecki:2006aa,Bojdecki:2007aa,Bojdecki:2007ad,Milos:2008aa,Milos:2007ab,Milos:aa,Milos:2008rc} is devoted to the study of systems with particles moving according to a symmetric $\alpha$-stable L\'evy motion and with critical branching (such systems will be referred to as the \textit{critical systems},  contrary to the \textit{subcritical systems} of this paper). The results therein split roughly into three classes depending on the dimension of the state space $\Rd$
\begin{itemize}
 \item ``low dimensions'' -- the system suffers local extinction. The direct study of the fluctuations of the rescaled occupation time does not make sense (except from the systems with immigration in \cite{Milos:2008rc})
 \item ``intermediate dimensions'' -- the limit has a simple spatial structure (Lebesgue measure) and a complicated temporal one (with long range dependence property).
 \item ``large dimensions'' -- the limit has a complicated spatial structure ($\SP$-valued random field) and a simple temporal one (process with independent increments).
\end{itemize}
We study the fluctuations of the rescaled occupation time process for systems with subcritical branching, which has never been done before. The main result is the functional limit theorem contained in Theorem \ref{thm:thm1}. The functional setting makes the result more interesting and much harder to prove. Moreover, we emphasise that the class of systems covered by this theorem is very broad as the restrictions imposed on the Markov family $(\eta_t,\mathbb{P}_x)_{t\geq0,x\in \Rd}$ are mild and natural - cf. Remark \ref{rem:ass} (this is contrary to the critical systems where the proofs rely heavily on the fine properties of a $\alpha$-stable L\'evy motion). In order to apprehend the theorem we turn our attention to three aspects of the result.\\
 Firstly, the theorem is a ``classical'' functional central limit theorem as the normalising factor is $F_T=T^{1/2}$ and the limit is Gaussian, namely a Wiener process. Therefore the temporal structure of the limit is simple - the increments of the process are independent. This contrasts sharply with the spatial structure, which is a $\SP$-valued random field of the form depending on the properties of the Markov family $(\eta_t,\mathbb{P}_x)_{t\geq0,x\in \Rd}$.\\
 Secondly, the subcriticality of the branching law is crucial for the long-term behaviour of the system. The limit if of the same nature in all dimensions, making this case much different from the critical systems, where the phenomenon of ``phase transition'' between ``intermediate'' and ``large dimensions'' is observed. The main reason for this are, roughly speaking, the properties of the movement of particles (recurrence vs. transience). In contrast, in the subcritical system the life-span of the family descending from one particle is short (its tail decays exponentially) hence the properties of the movement plays much smaller r\^ole. Moreover a particle hardly ever visits the same site multiple times which explains the similarity of the result to the one for the critical systems in ``large dimensions'' (i.e. the transient case). It also sheds some light on the origin of the temporal and spatial structures. If we consider two disjoint intervals which are far away, it is very likely that the increments of the occupation time are contributed by distinct families. This results in independent increments. On the other hand the life-span of a family is too short to ``even out the grains in'' which, in turn, gives rise to the complicated spatial structure. We stress that the subcriticality of the branching law influences the limit much more then the immigration. The results for an analogous immigration system but with critical branching \cite{Milos:2008rc} are much different and adhere to the scheme for critical systems presented above. See Remark \ref{rem:ass} for further explanation.  \\
Finally, notice that the systems considered in the paper do not suffer from the local extinction in ``low dimensions'' (due to immigration each set of positive measure is being populated in arbitrarily large times). As it was mentioned already, the limit is of the same nature as in the other dimensions cases.
It is interesting to compare this results with \cite{Bojdecki:2007ab}. ``High density'' technique applied there enables to study the occupation time fluctuations in ``low dimensions'', which led to a similar limits as in the case of ``intermediate dimensions''.\\
To make our paper more comprehensive we present also two illustrative examples in the results section. The example one presents perhaps the most important application of Theorem \ref{thm:thm1} to the system of particles moving according to a L\'evy motion. It can be regarded as a subcritical counterpart of the critical systems considered in earlier papers (it should be stressed however, that we admit much larger class of processes compared to a symmetric $\alpha$-stable motion previously considered). The resemblance to the "large dimensions" case is even more perceptible here - cf. Remark \ref{rem:small_imigraion_limit}. In the second example we consider a system with particles moving according to the Ornstein-Uhlenbeck process. It is intriguing because of the competition of particles attraction towards the origin (caused by the Ornstein-Uhlenbeck process) and their disappearance (caused by the subcriticality of the branching law) - cf. Remark \ref{rem:sub}\\
Recently, occupation time processes have been intensively. Further to the results mentioned previously \cite{Bojdecki:2007ac,Bojdecki:2008aa} presents results for systems with inhomogeneous starting distributions. One should also mention \cite{Birkner:2007aa,Birkner:2005aa} where similar problems are considered in discrete setting (lattice $\mathbb{Z}^d$). Interesting results were also obtained for superprocesses for example \cite{Iscoe:1986aa} and \cite{Hong:2005aa}. In \cite{Hong:2005aa} the authors consider a model very similar to ours, namely subcritical superprocess with immigration. They only examine the spacial structure (which is technically much easier) obtaining a similar result of Gaussian random field. One should also mention ~\cite{Ivanoff:1980aa} which was pioneering papers in the field of systems with immigration.\\
The proof technique is similar to the one from previous papers of Bojdecki et al. However, the subcritical case required developing new equations and dealing with a general Markov family. This required some of the technical arguments to be refined, as the fluctuations of the occupation time of systems with subcritical branching was studied for the first time.\\ 
The paper is organized as follows. In Section \ref{sec:res} we present assumptions and the general theorem (i.e. Theorem \ref{thm:thm1}). Next we give the examples mentioned above. Finally, Theorem \ref{thm:thm1} in Section \ref{sec:proof} and \ref{sec:labx}.

\section{Results} \label{sec:res}
\subsection{Notation}
Before presenting the  results announced in Introduction we clear out a few technical points. $\SP$ is a space of tempered distributions i.e. a nuclear space dual to the Schwartz space of rapidly decreasing functions $\SD$. The duality will be denoted by $\ddp{\cdot}{\cdot}$.\\ 
By $(\T{t})_{t\geq0}$, $A$ we will denote, respectively, the semigroup and the infinitesimal operator corresponding to the Markov family $(\eta_t,\mathbb{P}_x)_{t\geq0,x\in \Rd}$ presented in Introduction.  Sometimes instead of writing $\ev{}_x f(\eta_t)$ we write $\ev{f(\eta^x_t)}$.\\
For brevity of notation we also denote the semigroup
\begin{equation}
 \T{t}^Q f(x) := e^{-Qt}\T{t} f(x)
\end{equation} 
and the potential operator corresponding to it
\begin{equation}
 \mathcal{U}^Q f(x) = \intc{+\infty} \T{t}^Q f(x) \dd{t}.
\end{equation} 
In the whole paper 
\begin{equation}
 Q=V(1-2q),\label{def:Q}
\end{equation}
which intuitively denotes ``intensity of dying'' - recall that $V$ is the intensity of branching and $2q$ is the expected number of particles spawning from one particle. Clearly, subcriticality of the branching law implies $Q>0$.\\ 
Three kinds of convergence are used. The convergence of finite-dimensional distributions is denoted by $\rightarrow_{fdd}$. For a continuous, $\SP$-valued process $X=(X_t )_{t\geq0}$ and any $\tau > 0$ one can define an $\mathcal{S}'(\mathbb{R}^{d+1})$-valued random variable
\begin{equation}
 \ddp{\tilde{X}^\tau}{\Phi} = \intc{\tau}\ddp{X_t}{\Phi(\cdot,t)}\dd{t}, \label{eq:space-time-method}
\end{equation}
If for any $\tau > 0$ $\tilde{X}_n \rightarrow \tilde{X}$ in distribution, we say that the convergence in the space-time sense holds and denote this fact by $\rightarrow_i$. Finally, we consider the functional weak convergence denoted by $X_n\rightarrow_c X$. It holds if for any $\tau > 0$ processes $X_n = (X_n (t))_{t\in[0,\tau]}$ converge to $X = (X(t))_{t\in[0,\tau ]}$ weakly in $\mathcal{C}([0, \tau ], \SP)$ (in the sequel without loss of generality we assume $\tau = 1$). It is known that $\rightarrow_i$ and $\rightarrow_{fdd}$ do not imply each other, but either of them together with tightness implies $\rightarrow_c$. Conversely, $\rightarrow_c$ implies both $\rightarrow_i$ , $\rightarrow_{fdd}$.\\
By $c,c_1,\ldots,C,C_1,\ldots$ we will denote generic constants.

\subsection{General case} \label{sec:subcritical}
Firstly, we present the restrictions imposed on the Markov family $(\eta_t,\mathbb{P}_x)_{t\geq0,x\in \Rd}$. Not only are they mild and quite natural but also are easy to check in a concrete cases (see Section \ref{sec:examples}). First denote quadratic forms
\begin{equation}
  T_1(\varphi) := \intr U^Q\rbr{\varphi(x)U^Q\varphi(x)} \dd{x},\quad \varphi \in \SD, \label{eq:T1}
\end{equation} 
\begin{equation}
  T_2(\varphi) := \intr \int_{0}^{+\infty} U^Q \sbr{ \T{s}^Q\varphi(\cdot) \T{s}^Q U^Q\varphi(\cdot) }(x) \dd{s}\dd{x}, \quad \varphi \in \SD, \label{eq:T2}
\end{equation} 
also, slightly abusing notation, we will denote by $T_1$ and $T_2$ the \textit{bilinear forms} corresponding to them.

\paragraph{Assumptions 1}
\begin{enumerate}
\item[(A1)] We assume that the Markov family $(\eta_t,\mathbb{P}_x)_{t\geq0,x\in \Rd}$ is  almost uniformly stochastically continuous i.e.
\begin{equation}
\forall_n \sup_{x \in (-n,n)} \mathbb{P}_x(\eta_s,B(x,\epsilon)) \rightarrow 1, \quad \text{ as } s\rightarrow 0,
\end{equation} 
where $B(x,\epsilon)$ denotes a ball of radius $\epsilon$ with the center in $x$. Additionally, we assume that for any $x$ trajectories of process are almost surely bounded on any finite interval.
\item[(A1')] Instead of (A1) one can assume stronger but a more natural condition as follows. We assume that the Markov family $(\eta_t,\mathbb{P}_x)_{t\geq0,x\in \Rd}$ is uniformly stochastically continuous i.e.
\begin{equation}
 \sup_x \mathbb{P}_x(\eta_s,B(x,\epsilon)) \rightarrow 1, \quad \text{ as } s\rightarrow 0,
\end{equation} 
where $B(x,\epsilon)$ denotes a ball of radius $\epsilon$ with the center in $x$.
\item[(A2)] Denote by $D_A$ the domain of the infinitesimal operator $A$. We assume
\begin{equation}
 \SD \subset D_A.
\end{equation}
\item[(A3)] Let $\varphi\in \SD$. We assume that the semigroup $(\T{t}^\varphi)_{t\geq 0}$ given by
\begin{equation}
 \T{t}^\varphi f(x) = \ev{}_x \exp\cbr{-\intc{t} \varphi(\eta_s) \dd{s}} f(\eta_t),
\end{equation} 
is a Feller semigroup for any $\varphi$.
\item[(A4)] For any $\varphi\in \SD$
\begin{equation}
  T_1(\varphi) <+\infty,\quad T_2(\varphi) <+\infty. \label{ass:1}
\end{equation} 
\item[(A5)] For any $\varphi\in \SD$
\begin{equation}
 T^{3/2} \intr \T{T}^Q \varphi(x) \dd{x} \rightarrow 0. \label{ass:3}
\end{equation}
\end{enumerate}

\paragraph{Assumptions 2}
\begin{enumerate}
 \item[(A6)] There exists $\epsilon>0$ such that for any $\varphi$
\begin{equation}
              \intr \T{t}^Q \varphi(x) \dd{x} \leq c \rbr{1\wedge t^{-1-\epsilon}}.
             \end{equation} 
\item[(A7)] There exists $\epsilon>0$ such that for any $\varphi$ and for all $h,l$
\begin{equation}
 \intr \T{t}^Q\sbr{\T{h}^Q\varphi(\cdot) \T{l}^Q\varphi(\cdot) }(x) \dd{x} \leq c\rbr{1\wedge t^{-1-\epsilon}}.
\end{equation} 
\end{enumerate}
Now we are ready to formulate the theorem which is the main result of the paper.

\begin{thm} \label{thm:thm1}
Let $X_T$ be the rescaled occupation time fluctuations process given by  (\ref{def:occupation_process}). Assume that $F_T = T^{1/2}$ and assumptions (A1)-(A5) are fulfilled. Then
\begin{equation}
 X_T \rightarrow_{i} X,\quad \text{ and }\quad X_T \rightarrow_{fdd} X,
\end{equation} 
where $X$ is a generalized $\SP$-valued Wiener process with covariance functional
\begin{equation*}
 Cov(\ddp{X_t}{\varphi_1}, \ddp{X_s}{\varphi_2}) = H\rbr{s\wedge t} \rbr{T_1(\varphi_1, \varphi_2) + Vq T_2(\varphi_1, \varphi_2)}.
\end{equation*}
if, additionally, assumptions (A6)-(A7) are fulfilled then
\begin{equation}
 X_T \rightarrow_{c} X.
\end{equation} 
\end{thm}

\begin{rem} \label{rem:ass}
Assumptions (A1),(A2),(A3) are typical technical restrictions when dealing with Markov processes. We stress that they are mild and fulfilled easily by any "well-behaving" Markov process. Condition (A4) is natural, as it states only that the limits is well defined (if it is not fulfilled normalization larger then $F_T = T^{1/2}$ is required). Finally to analyse (A5) let us notice that 
\begin{equation}
 \intc{t} \intr \T{s}^Q \mathbf{1}_A(x) \dd{x} \dd{s},
\end{equation}
is the average number of particles in set $A$ for the system starting from the null measure. Intuitively, the aim of (A5) is to ``prevent gathering infinite number of particles'' in any set. Taking this explanation, (A5) seems to be too strong as $T^{-1-\epsilon}$ for any $\epsilon>0$ should be sufficient.\\
To sum up we state questions which raise naturally for further investigation. Firstly, the gap between (A5) and $T^{-1-\epsilon}$ is somehow unpleasant. A natural conjecture is that Theorem \ref{thm:thm1} holds also with this weaker condition. Another, less likely in the author's opinion, possibility is that the gap can be explained in probabilistic manner. Any result in this field would possibly give Theorem \ref{thm:thm1} even more elegant form. Secondly, assumptions (A3),(A4),(A5) impose a certain regime of behaviour on the system, in which the subcriticality suppress the contribution of the motion to the limit. By relaxing them the contribution of the motion ``increases''. Rough calculations suggest that this in turn results in an increase of the norming factor $F_T$ and in the limit with a complicated temporal part. However, with the motion playing larger r\^ole, this case it is not likely to be captured as generally and elegantly as in Theorem \ref{thm:thm1}. The reader is also referred to Remark \ref{rem:sub} for more detailed explanation in a particular example.
\end{rem}
\begin{rem}
 Assumptions (A4),(A5) are clearly technical. It is not obvious whether they are necessary. This question have not received enough attention yet, as the main goal of this paper was to identify the limits. Finding necessary conditions for tightness seems not to be easy, however.
\end{rem}
\begin{rem}
As it was mentioned in Introduction the limit is a $\SP$-valued Wiener process with a simple time structure and a complicated temporal one in all dimensions. This result resembles the result for the system with critical branching in large dimensions. The main reason of this is a short (exponentially-tailed) life-span of a family  descending from one particle. On the one hand it leads to independent increments in the limit (as there are no ``related'' particles in long term). On the other hand the movement is ``not strong enough'' to smooth out the space structure.\\ 
Another remarkable, yet not such unexpected, feature is that the limit can be obtained for "low dimensions". Due to immigration the system no longer suffers from local extinction and the limit can be obtained without special techniques, like high density limits of ~\cite{Bojdecki:2007ab}. 
\end{rem}
\begin{rem}
 Analogous systems but with critical branching were studied in \cite{Milos:2008rc}. The results there are much different from Theorem \ref{thm:thm1} and adhere to the scheme observed for other critical systems (as described in Introduction). This proves clearly that the subcriticality influences systems much more then the immigration, at least with respect to the limit behaviour of the fluctuations of the occupation time.
\end{rem}

\subsection{Examples} \label{sec:examples}
The theorem in the previous section is quite abstract. Now we will present two illustrative examples.
\paragraph*{L\'evy motion} Recall the description of the system $N$ from Introduction. In this example the movement of particles will be given by a L\'evy process - we keep the notation, by $(\eta_t)_t$ we denote the L\'evy motion starting from $0$. Its characteristic function is
\begin{equation}
 \ev{} \sbr{ e^{iz \eta_t}} =  \exp \Bigg(t \Psi(z)\Bigg),
\end{equation} 
where $\Psi$ is the L\'evy-Khinchine exponent i.e.
\begin{equation}
 \Psi(z) = i \ddp{z}{a} - \frac{1}{2} \ddp{Kz}{z} +  \int_{\mathbb{R^d}\backslash\{0\}} \big( e^{i\ddp{z}{x} }-1 -i\ddp{\theta}{x}  \mathbf{1}_{|x|<1}\big)\,\mu(\dd{x}),\quad x\in \Rd,
\end{equation} 
where $a\in \Rd$ (drift term), $K$ is non-negative defined $n\times n$ matrix (covariance of the Gaussian part) and $\mu$ is a (spectral) measure obeying condition $\int_{\mathbb{R^d}\backslash\{0\}} (x^2 \wedge 1) \dd{x} <+\infty$.\\
Let us now check that the L\'evy motion fulfils Assumptions 1. It is a space homogeneous process hence to check (A1') it suffices to show that $\eta_t \rightarrow^P 0$ which follows directly from the characteristic function. (A2) is slightly more difficult - let us take $\varphi\in \SD$; one can check that 
\begin{equation*}
 \frac{\widehat{\T{t}\varphi}(z) - \widehat{\varphi}(z) }{t} \rightarrow^{\mathcal{L}_1} i \Psi(z)\widehat{\varphi}(z).
\end{equation*}
which implies
\begin{equation*}
 \frac{{\T{t}\varphi}(z) - {\varphi}(z) }{t} \rightarrow^{sup} i \widehat{\Psi(z)\widehat{\varphi}(z)}.
\end{equation*}
Hence $\varphi\in D_A$.\\
We skip the proof of (A3) which is as an easy consequence of the space homogeneity and Lebesgue's dominated convergence theorem.\\ 
Recall that $\lambda$ is an invariant measure of the L\'evy motion. We have
\begin{equation}
 T_1(\varphi) = \intr U^Q\rbr{\varphi(x)U^Q\varphi(x)} \dd{x} = \frac{1}{Q}\intr \varphi(x)U^Q\varphi(x) \dd{x}. 
\end{equation} 
It can be checked that $\widehat{U^Q\varphi}(z) = \widehat{\varphi}(z) (Q-\Psi(z))^{-1}$. Applying the Fourier transform to \eqref{eq:T1} we obtain
\begin{equation}
 T_1(\varphi)  = \frac{1}{(2\pi)^d}\frac{1}{Q}\intr \frac{|\widehat{\varphi}(z)|^2}{Q-\Psi(z)} \dd{z}.
\end{equation} 
$T_2$ can be treated in a similar way
\begin{equation}
 T_2(\varphi) = \frac{1}{(2\pi)^d}\frac{1}{Q}\intr \frac{|\widehat{\varphi}(z)|^2}{ (Q-\Psi(z))^2 } \dd{z}.
\end{equation} 
The real part of $\Psi$ is non-positive hence clearly both $T_1(\varphi)$ and $T_2(\varphi)$ are finite, therefore (A4) holds.\\
The assumption (A5) follows easily from calculations below
\begin{equation*}
 \intr \T{T}^Q \varphi(x) \dd{x} = e^{-QT}\intr \T{T} \varphi(x) \dd{x} = c e^{-QT}. 
\end{equation*}
Finally (A6)-(A7), can be proved in the same way. Utilising Theorem \ref{thm:thm1} we obtain
\begin{thm}
Let $X_T$ be the occupation time fluctuation process given by  (\ref{def:occupation_process}) for a system of particles moving according to a L\'evy motion. Then
\[X_T \rightarrow_c X,\quad  \text{ as } T\rightarrow +\infty,\]
where $X$ is a generalized $\SP$-valued Wiener process with covariance functional
\begin{multline}
 	Cov\left( \ddp{X_s}{\varphi_1}, \ddp{X_t}{\varphi_2} \right) = \\
(s\wedge t) \frac{H}{Q} \frac{1}{(2\pi)^d}
\intr \:\left( \frac{1}{Q-\Psi(z)} + \frac{Vq}{\rbr{Q-\Psi(z)}^2  } \right) \widehat{\varphi}_1(z) \overline{\widehat{\varphi}_2(z)} \dd{z},\quad \varphi_1,\varphi_2\in\mathcal{S}\left(\mathbb{R}^{d}\right). \label{eq:cov_levy}
\end{multline}
\end{thm}
\begin{rem} \label{rem:small_imigraion_limit}
Formally the result resembles the result for the critical branching systems in "large dimensions". Indeed, by converging with branching law to a critical one (i.e. $q\rightarrow0$) and decreasing intensity of immigration (i.e. $H\rightarrow 0$) appropriately in the rhs of the expression above, one gets
\begin{equation}
(s\wedge t) \frac{1}{(2\pi)^d}
\intr \:\left( \frac{1}{-\Psi(z)} + \frac{Vq}{{\Psi^2(z)}  } \right) \widehat{\varphi}_1(z) \overline{\widehat{\varphi}_2(z)} \dd{z},
\end{equation} 
which exactly the limit in theorem \cite[Theorem 2.1]{Bojdecki:2006aa} (with $\Psi(z) = -z^\alpha$ for the symmetric $\alpha$-stable L\'evy motion considered there). The question if this convergence has any probabilistic meaning is natural but has not been addressed yet.
\end{rem}
\begin{rem}
 Denote the spacial part of the limit by $Y$. It easy to notice that it is a homogeneous (generalized) Gaussian random field. The measure
\begin{equation}
 \mu(\dd{z}) :=  \rbr{\frac{1}{Q-\Psi(z)} + \frac{Vq}{\rbr{Q-\Psi(z)}^2  } } \dd{z},
\end{equation} 
is called the spectral measure of $Y$. It is well-known (see e.g. \cite[Preposition 1]{Karczewska:2001gd}) that $Y$ is "classical" i.e. is function-valued random field if and only if its spectral measure is finite. For the system considered in this section this translates to the condition
\begin{equation}
 \intr \frac{1}{Q-\Psi(z)} \dd{z} < +\infty. 
\end{equation} 
In the most important case when the particles move according to the symmetric $\alpha$-stable L\'evy motion the condition is true if and only if $d=1$ and $\alpha\in (1,2]$.
\end{rem}

\paragraph*{Ornstein-Uhlenbeck process} In this example the movement of particles is governed by the Ornstein-Uhlenbeck process. The Ornstein-Uhlenbeck process is the solution of stochastic equation
\begin{equation}
 dX_t = -\theta X_t \,\dd{t} + \sigma\, dW_t, \quad \theta>0,\sigma\neq 0.
\end{equation}
Its semigroup is given by
\begin{equation}
 \T{t}f(x) = \rbr{\mathcal{S}_{ou(t)} f}(xe^{-\theta t}) \label{eq:ou-semigroup}
\end{equation} 
where $\mathcal{S}$ is the semigroup of Wiener process and $ou(t) = (1-e^{-2\theta t})/(2\theta),\quad OU(t) = (e^{2\theta t}-1)/(2\theta)$. \\
Assumptions (A1),(A2) i (A3) can be checked easily from the following representation of the Ornstein-Uhlenbeck process
\begin{equation*}
 X^x_t= x e^{-\theta t} + {\sigma\over\sqrt{2\theta}}W(e^{2\theta t}-1)e^{-\theta t}. 
\end{equation*}
Recall definition \ref{def:Q}. We assume also that $Q>\theta$ (this assumption is crucial and will be explained later in Remark \ref{rem:sub}). We have 
\begin{equation*}
 \intr \T{t}^{Q} f(x) \dd{x} = e^{-( Q - \theta) t} \intr f(x)\dd{x}.
\end{equation*}
Using this equation (A4)-(A6) can be easily verified.\\
Using the Fourier transform we can calculate $T_1$ and $T_2$ in more explicit form. Namely, by \eqref{eq:ou-semigroup} and the fact that the Lebesgue measure is an invariant measure of $\mathcal{S}$ we have 
\begin{equation*}
 T_1(\varphi) = \intr \mathcal{U}^Q\rbr{\varphi(x)\mathcal{U}^Q\varphi(x) } \dd{x} = \frac{1}{Q} \intr \varphi(x)\mathcal{U}^Q\varphi(x) \dd{x}.
\end{equation*}
Using the Fourier transform we get 
\begin{equation*}
 T_1(\varphi) = \frac{1}{(2\pi)^d}\frac{1}{Q} \inti e^{-(Q-\theta)t} \intr \widehat{\varphi}(z) e^{-OU(t)|z|^2} \overline{\widehat{\varphi}(e^{\theta t} z)} \dd{z}
\end{equation*}
Similar calculations for $T_2$ give
\begin{multline*}
 T_2(\varphi) = \\
\frac{1}{(2\pi)^d}\frac{1}{Q} \inti \inti  e^{-(Q-\theta)(2s+u)} \intr \widehat{\varphi}(e^{\theta s}z) e^{-2OU(s)|z|^2} \widehat{\varphi}(e^{\theta t} z) e^{-OU(u)|e^{\theta s}z|^2} \overline{\widehat{\varphi}(e^{\theta (s+u)}z)}\dd{z} \dd{u} \dd{s}.
\end{multline*}
(Recall that quadratic forms $T_1$ and $T_2$ induce corresponding bilinear forms).
Assumptions (A6), (A7) can be easily verified this entitles us to use Theorem \ref{thm:thm1}.\\
\begin{thm}
Let $X_T$ be the rescaled occupation time fluctuation process given by  (\ref{def:occupation_process}) for a system of particles moving according to the Ornstein-Uhlenbeck process. Assume that $Q>\theta$. Then
\[X_T \rightarrow_c X,\quad  \text{ as } T\rightarrow +\infty,\]
where $X$ is a generalized $\SP$-valued Wiener process with covariance functional
\begin{equation*}
 	Cov\left( \ddp{X_s}{\varphi_1}, \ddp{X_t}{\varphi_2} \right) = 
(s\wedge t) \frac{H}{Q} \frac{1}{(2\pi)^d}
 \rbr{T_1(\varphi_1, \varphi_2) + Vq T_2(\varphi_1,\varphi_2)\!\!\frac{}{}},\quad \varphi_1,\varphi_2\in\mathcal{S}\left(\mathbb{R}^{d}\right). \nonumber
\end{equation*}
\end{thm}
\begin{rem} \label{rem:sub}
 This example is interesting because we can observe struggle of two antagonistic forces. One is the ``exponential attraction'' of particles from the whole space to the proximity of $0$ by the Ornstein-Uhlenbeck process the other is dying out of particles because of the subcriticalty of the branching law. More precisely, denote $\varphi = \mathbf{1}_{B(0,r)}$ then
\begin{equation*}
 \intr \T{t}^Q\varphi(x) \dd{x} = e^{-Qt}\intr  \rbr{\mathcal{S}_{ou(t)} \varphi}(xe^{-\theta t}) \dd{x} = e^{-(Q-\theta) t} |B(0,r)|.
\end{equation*}
is the average number of particles in the ball $B(0,r)$ for the subcritical system (without immigration) starting out from the homogeneous Poisson field. The condition $Q>\theta$ can now be  easily interpreted - the subcriticality is ``strong enough'' to prevent gathering of particles (near $0$).\\ 
This observation raises a natural question what happen when $Q=\theta$ i.e. when the forces are in the perfect balance. Rough calculations suggest that norming factor is greater ( $F_T=T$) and the properties of the motion affect the temporal part of the limit (it is not longer process with independent increments).
\end{rem}

\section{Proofs} \label{sec:proof}
\subsection{Scheme} \label{sec:scheme}
To make the proof clearer we present a general scheme here and defer details to separated sections. Although the processes $X_T$ are signed-measure-valued it is convenient to regard them as processes with values in $\SP$. In this space one may employ a space-time method introduced by \cite{Bojdecki:1986aa} which together with Mitoma’s theorem constitute a powerful technique in proving weak, functional convergence.
\paragraph*{Convergence} From now on we will denote by $\tilde{X}_T$ ($\tau=1$) a space-time variable (recall \eqref{eq:space-time-method} with $\tau=1$) defined for $X_T$. To prove convergence of $\tilde{X}_T$ we will use the Laplace functional
\begin{equation}
 L_T(\Phi) = \ev{e^{-\ddp{\tilde{X}_T}{\Phi}}},\quad \Phi\in \mathcal{S}(\mathbb{R}^{d+1}), \Phi\geq 0.
\end{equation}
For the limit process $X$ denote
\begin{equation*}
 L(\Phi) = \ev{e^{-\ddp{\tilde{X}}{\Phi}}},\quad \Phi\in \mathcal{S}(\mathbb{R}^{d+1}), \Phi \geq 0.
\end{equation*}
Once we establish convergence 
\begin{equation}
 L_T(\Phi) \rightarrow L(\Phi),\quad \text{ as } T\rightarrow +\infty, \forall_{\Phi\in \mathcal{S}(\mathbb{R}^{d+1}), \Phi \geq 0}. \label{lim:laplace}
\end{equation} 
we will obtain week convergence $\tilde{X}_T\Rightarrow \tilde{X}$ and consequently $X_T\rightarrow_i X$. Two technical remark should be made here. We consider only non-negative $\Phi$. The procedure how to extend the convergence to any $\Phi$ is explained in \cite[Section 3.2]{Bojdecki:2006ab}. Another issue is the fact that $\ddp{\tilde{X}_T}{\Phi}$ is not non-negative (which is a usual condition to use the Laplace transform). The usage of the Laplace transform in this paper is justified by the special (Gaussian) form of the limit. For more detailed explanation one can check also \cite[Section 3.2]{Bojdecki:2006ab}.\\
As explained in \cite{Bojdecki:2007aa} due to the special form of the Laplace transform convergence \eqref{lim:laplace} implies also finite-dimensional convergence.\\
Detailed calculations for this part of scheme will be conducted in Section \ref{sec:laplace} and Section \ref{sec:lab}.
\paragraph*{Tightness} Using additional assumptions the tightness can be proved utilizing the Mitoma theorem \cite{Mitoma:1983aa}, which states that tightness of $\cbr{X_T}_T$ with trajectories in $C([0, 1], \SP)$ is equivalent to tightness of $\ddp{X_T}{\phi}$, in $\mathcal{C}([0,\tau],\mathbb{R})$ for every $\phi \in \mathcal{S}(\Rd)$. We adopt a technique introduced in \cite{Bojdecki:2006aa}. Recall a classical criterion \cite[Theorem 12.3]{Billingsley:1968aa}, i.e. a process $\phi \in \mathcal{S}(\Rd)$ is tight if for all $t,s\geq0$
\begin{equation}
 \ev{(\ddp{X_T(t)}{\varphi}, \ddp{X_T(s)}{\varphi})^4} \leq C (t-s)^2. \label{ineq:tightness}
\end{equation}
Following the scheme in \cite{Bojdecki:2006aa} we define a sequence $(\psi_n)_n$ in $\mathcal{S}(\mathbb{R})$, and $\chi_n(u) = \int_u^1 \psi_n(s) \dd{s}$ in a such way that 
\[
\psi_n\rightarrow \delta_t - \delta_s,
\]
\begin{equation}
 0\leq \chi_n \leq \mathbf{1}_{[s,t]}.\label{tmp:chi-n}
\end{equation}
Denote $\Phi_n=\varphi\otimes\psi_n$. We have
\[
 \lim_{n\rightarrow +\infty} \ddp{X_T}{\Phi_n} = \ddp{X_T(t)}{\varphi} - \ddp{X_T(s)}{\varphi}
\]
 thus by Fatou's lemma and the definition of $\psi_n$ we will obtain (\ref{ineq:tightness}) if we prove ($C$ is a constant independent of $n$ and $T$) that
\[
 \ev{\ddp{\tilde{X}_T}{\Phi_n}^4} \leq C (t-s)^2.
\]
From now on we fix an arbitrary $n$ and denote $\Phi:=\Phi_n$ and $\chi := \chi_n$. 
By properties of the Laplace transform we have
\[
 \ev{\ddp{\tilde{X}_T}{\Phi}^4} = \left.\frac{d^4}{d\theta^4}\right|_{\theta = 0}\!\! \ev{\eexp{-\theta \ddp{\tilde{X}_T}{\Phi}}}
\]
Hence the proof of tightness will be completed if we show 
\begin{equation}
 \left. \frac{d^4}{d\theta^4}\right|_{\theta = 0}\!\! \ev{\eexp{-\theta \ddp{\tilde{X}_T}{\Phi}}} \leq C(t-s)^2. \label{eq: laplace-tighness-end}
\end{equation}
Further calculations are deferred to Sections \ref{sec:tightness_plus} and \ref{sec:labt}.

\subsection{Auxiliary facts and one-particle equation}
In this Section we will prove an equation for one particle which will play a key r\^ole in the rest of the proof. Before that we recall a general Feynman-Kac formula which is crucial for its proof.\\
Let $A$ be a (unbounded) linear operator with domain $D_A$. We define a problem
	\begin{equation}
		\left \lbrace 
		\begin{array}{cc}
 \frac{\partial}{\partial t} w(t,x) = A w(t,x) + c(x)w(t,x),\\
w(0,x) = f(x), \quad \quad\quad \quad\quad \quad\quad \quad \:\:
		\end{array} 
		\right . \label{def:problem}
	\end{equation}
where $w(\cdot,t),f\in D_A$.

\begin{prop}[Feynman-Kac formula] \label{prop:feynamn-kac}
Let $(X_t,\mathbb{P}_x)$ be a uniformly stochastically continuous Markov family (cf. assumption (A1)) with values in $\Rd$ with infinitesimal operator $A$. Assume also then $c$ is a uniformly continuous and bounded. Then 
\begin{equation*}
 w(t,x) = \mathbb{E}_x\exp \cbr{\intc{t} c(X_s) \dd{s}}f(X_t),\quad{t\geq0},x \in E,
\end{equation*}
is a solution of \eqref{def:problem}. It is the only solution in the class of functions $w$ such that $\sup_{x} |x(x,t)| \leq e^{\alpha t},$ $\forall_t$ for $\alpha\in \mathbb{R}$.
\end{prop}
Recall that $F$ denotes generating function of the branching law \eqref{eq:generating}. We define $G(s) = F(1-s) - (1-s)$ so in our case
\begin{equation}
 G(s) = qs^2 + (1-2q)s. \label{def:G}
\end{equation}
Behavior of the system starting off from a single particle at $x$ is described by the function
\begin{equation}
v_{\Psi}\left(x,r,t\right)=1-\mathbb{E}\exp\left\{ -\int_{0}^{t}\left\langle N_{s}^{x},\Psi\left(\cdot,r+s\right)\right\rangle \dd{s}\right\}, \Psi\geq0 \label{def:v},
\end{equation}
where $N_s^x$ denotes the empirical measure of the particle system with the initial condition $N_0^x = \delta_x$. More precisely $N^x$ is a system in which particles evolve according to the dynamics described in Introduction but without immigration.\\
The lemma gives the announced equation
\begin{lem} \label{lem:equation}
 Assume that $\Psi\geq 0$ and assumptions (A1)-(A3) are fulfilled then
\begin{equation}
 v_\Psi \in [0,1], \label{eq:vin[0,1]}
\end{equation} 
and $v_\Psi$ satisfies equation
\begin{equation}
v_{\Psi}\left(x,r,t\right)=\int_{0}^{t}\mathcal{T}_{t-s}\left[\Psi\left(\cdot,r+t-s\right)\left(1-v_{\Psi}\left(\cdot,r+t-s,s\right)\right)-VG\left(v_{\Psi}\left(\cdot,r+t-s,s\right)\right)\right]\left(x\right)\dd{s}.\label{eq:v-integral}
\end{equation}
\end{lem}
Formally, this equation is the same as \cite[(3.22)]{Bojdecki:2006ab}. We have to refine the proof as in this paper we consider a more general case hence 
\begin{proof}
\eqref{eq:vin[0,1]} follows directly from the definition \eqref{def:v}.\\
Now we proceed to the proof of \eqref{eq:v-integral}. Denote
\begin{equation*}
 w_\Psi(x,r,t) := 1-v_\Psi(x,r,t)
\end{equation*}
In the first step we expand (A3) to a slightly more general semigroup. For $\Psi\in \mathcal{S}(\mathbb{R}^{d+1})$, $r\geq 0$ define $\T{}^{\Psi,r}$
\begin{equation*}
 \T{t}^{\Psi,r} f(x) =  \mathbb{E}\exp\left\{ -\int_{0}^{t}\left\langle \eta^{x}_s,\Psi\left(\cdot,r+s\right)\right\rangle \dd{s}\right\} f(\eta^{x}_t)
\end{equation*}
We claim that $\T{}^{\Psi,r}$ is also Feller.\\ 
Define $\Psi_n(x,t) = \sum_{k=1}^{n} \Psi(x,t_k) \mathbf{1}_{[t_{k-1},t_k)}(t)$ where $t_k = tk/n$. Inductive argument (with respect to $n$) implies easily that $\T{}^{\Psi_n,r}f(x)$ is continuous when $f$ is continuous. Indeed one can write
\begin{equation}
 \T{}^{\Psi_n,r}f(x) = \ev{}_x \exp\cbr{-\int_0^{t_1} \Psi(\eta_s,r+t_1) }\exp\cbr{-\int_{t_1}^{t_n} \Psi_n(\eta_s,r+s) \dd{s}}.
\end{equation} 
Using the Markov property we have
\begin{equation}
 \T{}^{\Psi_n,r}f(x) = \ev{}_x \exp\cbr{-\int_0^{t_1} \Psi(\eta_s,r+ t_1) } \T{t-t_1}^{{\Psi}_n,r+t_1}f(\eta_{t_1}) =\T{t_1}^{\Psi(\cdot, t_1+r)} \T{t-t_1}^{{\Psi}_n,r+t_1}f(x).
\end{equation} 
By induction we can assume that $\T{t-t_1}^{{\Psi}_n,r+t_1}f(x)$ is continuous and by the Feller property of $\T{t_1}^{\Psi(\cdot, t_1+r)}$ (assumption (A3)) we get asserted claim. It is obvious that $\T{t}^{\Psi_n,r}f \rightarrow \T{}^{\Psi,r}f$ uniformly hence $\T{}^{\Psi,r}f$ is continuous.\\
 In the next step we will prove that this fact implies the continuity of $w_\Psi$. Denote by $\mathbf{T}$ the space of ancestor trees i.e. a space of binary trees with nodes and leafs labeled by the splitting and death times of particles respectively. The splitting dynamics described in Introduction induces a probability measure on them - $\nu$. Let us notice that in our case the trees are finite almost surely. For a given ancestor tree $\tau$ we can calculate $w_{\tau,\Psi}$ given by
\begin{equation*}
 w_{\tau,\Psi}(x,r,t) := \ev{}\exp \cbr{-\intc{t} \ddp{M^x_s}{\Psi(\cdot,r+s)} \dd{s}},
\end{equation*}
where $M^x$ is the branching particle system with the branching dynamics encoded by $\tau$. Let $|\tau|$ denotes the height of $\tau$. By induction with respect to the height of the tree we can prove that $w_{\tau,\Psi}$ is continuous. For trees of height $1$ it is obvious from the continuity of $z$. Let $\tau$ be a tree such that $n=|\tau|$. Removing the root  splits $\tau$ into two sub-trees $\tau_1$, $\tau_2$. By $t_1$ we denote the label of the root i.e. the time of the first branching (or death if the root is a leaf). If $t_1>t$ the continuity if obvious hence we are remain only with the situation when $t_1<t$. One can write 
\begin{equation*}
 w_{\tau, \Psi}(x,r,t) = \ev{}\exp \cbr{-\intc{t_1} \ddp{\eta^x_s}{\Psi(\cdot,r+s)} \dd{s}}\rbr{w_{\tau_1,\Psi}(\eta^x_{t_1},r+t_1,t-t_1) w_{\tau_2,\Psi}(\eta^x_{t_1},r+t_1,t-t_1) }.
\end{equation*}
Now continuity of $w_{\tau, \Psi}(\cdot,r,t)$ follows from the induction hypothesis and Feller property of $\T{}^{\Psi,r}$. Further, it can be easily proved that  $w_{\tau,\Psi}$ is in fact continuous as a function of three variables. This property infer continuity of $w_{\Psi}$ which is justified by the formula
\begin{equation*}
 w_\Psi(x,r,t) = \int_{\mathbf{T}} w_{\tau, \Psi}(x,r,t) \nu(\dd{\tau}).
\end{equation*}
and  Lebesgue's dominated convergence theorem (recall $0\leq w_{\tau,\Psi} \leq 1$ and $\nu$ is a probability measure). \\Conditioning on the time of the first branching we get
\begin{align*}
w\left(x,r,t\right)= & e^{-Vt}\mathbb{E}\left(-\int_{0}^{t}\Psi(\eta_{s}^{x},r+s)\dd{s}\right)\\
 & +V\int_{0}^{t}e^{-Vs}\mathbb{E}\exp\left(-\int_{0}^{s}\Psi(\eta_{u}^{x},r+u)\dd{u}\right)F\left(w\left(\eta_{s}^{x},r+s,t-s\right)\right).\end{align*}
Let us introduce functions $h$ i $k$
\[
h\left(x,r,t\right):=e^{-Vt}\mathbb{E}\exp\left\{ -\int_{0}^{t}\Psi\left(\eta_{s}^{x},r+s\right)\dd{s}\right\} ,\]
\begin{equation*}
k_s\left(x,r,t\right):=e^{-Vt}\mathbb{E}\exp\left(-\int_{0}^{t}\Psi(\eta_{u}^{x},r+u)
\dd{u}\right)F\left(w\left(\eta_t^{x},r+t,s\right)\right). 
\end{equation*}
Now $w$ can be written as
\begin{equation}
 w\left(x,r,t\right)= h\left(x,r,t\right)+V\int_{0}^{t} k_s\left(x,r,t-s\right)\dd{s}. \label{tmp:gora}
\end{equation} 
The crucial step of the proof is application of Feynman-Kac formula. Assume for a moment that the Markov family fulfils (A1') instead of (A1). Let $\Theta\in \mathcal{S}(\mathbb{R}^{d+1})$ and define 
\begin{equation*}
l_\Theta(x,r,t)=\mathbb{E}\exp\left(-\int_{0}^{t}\Psi(\eta_{u}^{x},r+u) \dd{u}\right) \Theta(\eta_t^x, r+t).
\end{equation*}
Assumptions (A1'), (A2) assert that we can use Proposition \ref{prop:feynamn-kac} (one have to prove that $\Theta$ belongs to the domain of the infinitesimal operator of Markov family $t\rightarrow (X^x_t,r+t)$ - we skip this simple step) hence $l_\Theta(x,r,t)$ is solution of
	\[
		\left \lbrace 
		\begin{array}{cc}
 \frac{\partial }{\partial t}l_\Theta(x,r,t) = \rbr{\Delta_\alpha + \frac{\partial}{\partial r}  - \Psi(x,r)} l_\Theta(x,r,t),\\
l_\Theta(x,r,0) = \Theta(x,r). \quad \quad\quad \quad\quad \quad \quad \quad\quad \quad \:\:
		\end{array} 
		\right . 
	\]
Let us denote  
\begin{equation}
 k_\Theta(x,r,t) = e^{-Vt} l_\Theta(x,r,t). \label{eq:kTheta}
\end{equation} 
Direct computations yield
	\[
		\left \lbrace 
		\begin{array}{cc}
 \frac{\partial }{\partial t}{k}_\Theta(x,r,t) = \rbr{\Delta_\alpha + \frac{\partial}{\partial r}}{k}_\Theta(x,r,t)  - \rbr{\Psi(x,r) +V } {k}_\Theta(x,r,t),\\
{k}_\Theta(x,r,0) = \Theta(x,r). \qquad \qquad\qquad \qquad\qquad \qquad \qquad \quad\quad \quad \:\:
		\end{array} 
		\right . 
	\]
This is an evolution equation which has an integral form 
\begin{equation}
 k_\Theta(x,r,t) = \T{t} \Theta(x,r+t) - \intc{t} \T{t-s} \sbr{\rbr{\Psi(\cdot,r +t -s ) +V } {k}_\Theta(\cdot,r +t -s, s)}(x) \dd{s}. \label{tempa88}
\end{equation}
Now define $\tau_n = \inf\{t:|\eta^x_t|>n\}$ and processes $\eta^{n,x}_t := \eta^x_{\tau_n \wedge t}$. Clearly they are Markov and each of them fulfils (A1) and (A2). We know so far that
\begin{equation}
 k^n_\Theta(x,r,t) = \T{t}^n \Theta(x,r+t) - \intc{t} \T{t-s}^n \sbr{\rbr{\Psi(\cdot,r +t -s ) +V } k^n_\Theta(\cdot,r +t -s, s)}(x) \dd{s}, \label{temp11}
\end{equation}
where $\T{}^n,k^n_\Theta$ denote respectively semigroup and equation \eqref{eq:kTheta} defined for Markov process $\eta^{x,n}$. It easy to show that $k^n\rightarrow k$ (point-wise) and consequently using Lebesgue dominated convergence theorem show that \eqref{tempa88} is fulfilled for any Markov family satisfying (A1).\\
Clearly  $F(w(\cdot,\cdot, s))$ is continuous hence there exists a sequence $(\Theta_n)_n$, $\Theta_n\in \SD$ convergent uniformly to $F(w(\cdot,\cdot, s))$. Applying this to definition \eqref{eq:kTheta} we obtain point-wise convergence
\begin{equation*}
 k_{\Theta_n}(x,r,t) \rightarrow k_s(x,r,t).
\end{equation*}
Now we use dominated Lebesgue's convergence theorem ($k_{\Theta_n}\leq \sup \Theta_n <c $) to the right side of \eqref{tempa88}
\begin{equation}
 k_s(x,r,t) = \T{t} F(w(x,r+t,s)) - \intc{t} \T{t-s} \sbr{\rbr{\Psi(\cdot,r + t -s ) +V } {k}_s(\cdot,r +t -s, s)}(x) \dd{s}.
\end{equation}
Analogously
\begin{equation}
 h(x,r,t) = 1 - \intc{t} \T{t-s} \sbr{\rbr{\Psi(\cdot,r +t -s ) +V } {h}(\cdot,r +t -s, s)}(x) \dd{s}.
\end{equation}
We put the obtained equations in \eqref{tmp:gora}
\begin{multline*}
  w\left(x,r,t\right)= 1 - \intc{t} \T{t-s} \sbr{\rbr{\Psi(\cdot,r +t -s ) +V } {h}(\cdot,r +t -s, s)}(x) \dd{s}+\\V\int_{0}^{t} \T{t-s} F(w(x,r+t-s,s)) \dd{s}-\\
V\int_{0}^{t} \intc{t-s} \T{t-s-u} \sbr{\rbr{\Psi(\cdot,r +t - s - u ) +V } {k}_s(\cdot, r + t - s - u, u)}(x) \dd{u} \dd{s}.
\end{multline*}
We substitute $u\rightarrow u -s $ and change the order of integration
\begin{multline*}
 w\left(x,r,t\right)= 1 - \intc{t} \T{t-s} \sbr{\rbr{\Psi(\cdot,r +t -s ) +V } {h}(\cdot,r +t -s, s)}(x) \dd{s}+\\V\int_{0}^{t} \T{t-s} F(w(x,r+t-s,s)) \dd{s}-\\
\int_{0}^{t} \T{t-u} \rbr{\Psi(\cdot,r +t - u ) +V } \sbr{V\intc{u}   {k}_s(\cdot, r + t - u, u-s)\dd{s}}(x)  \dd{u}.
\end{multline*}
Finally we apply \eqref{tmp:gora} to the second and fourth term
\begin{multline*}
 w\left(x,r,t\right)= 1 - \intc{t} \T{t-s} \sbr{\rbr{\Psi(\cdot,r +t -s ) +V } {w}(\cdot,r +t -s, s)}(x) \dd{s}+\\V\int_{0}^{t} \T{t-s} F(w(x,r+t-s,s)) \dd{s}.
\end{multline*}
Recall that $1-w = v_{\Psi}$,. Finally trivial computations yield asserted \eqref{eq:v-integral}.\\
\end{proof}
We consider the case of subcritical branching ($q<1/2$) in \eqref{eq:generating}. Recall equation \eqref{def:G} and $Q = V(1-2q)$, putting this to equation \eqref{eq:v-integral} gives
\begin{multline}
 v_{\Psi}\left(x,r,t\right)=\int_{0}^{t}\mathcal{T}_{t-s}\left[\Psi\left(\cdot,r+t-s\right)\left(1-v_{\Psi}\left(\cdot,r+t-s,s\right)\right)- \right. \\ \left. Qv_{\Psi}\left(\cdot,r+t-s,s\right) - Vq v_{\Psi}\left(\cdot,r+t-s,s\right)^2 \right]\left(x\right)\dd{s}. \label{eq:equation-sub}
\end{multline}
$v_\Psi$ is quite cumbersome to deal with hence we will approximate it with $\tilde{v}_\Psi$ defined in the following way
\begin{equation}
 \tilde{v}_{\Psi}(x,r,t) = \intc{t} \T{t-s}^Q \Psi(\cdot,r+t-s) \dd{s}, \quad \Psi \in \mathcal{S}(\mathbb{R}^{d+1} ), x\in \Rd, r,t\geq0.\label{sol:tv}
\end{equation} 
It can be easily checked that this function fulfills the equation
\begin{equation}
  \tilde{v}_{\Psi}(x,r,t) = \intc{t} \T{t-s} \left[ \Psi(\cdot, r+t-s) -  Q  \tilde{v}_{\Psi}(\cdot, r+t-s,s) \right](x) \dd{s}. \label{eq:tildev}
\end{equation} 
Intuitively $\tilde{v}_\Psi$ was obtained by dropping quadratic terms in (\ref{eq:equation-sub}) which ``do not play r\^ole'' when $\Psi$ is small. The quality of the approximation is expressed in terms of function $u$
\begin{equation}
 u_{\Psi} := \tilde{v}_{\Psi} - v_{\Psi}. \label{eq:def-u}
\end{equation} 
We have
\begin{lem}
  Let $\Psi\geq0$, then $u_\Psi$ satisfies the equation
\begin{equation}
  u_{\Psi} (x,r,t) = \intc{t}  \T{t-s}^Q \sbr{\Psi(\cdot, r+t-s) v_{\Psi}(\cdot, r+t-s,s) + V q v_{\Psi}^2(\cdot, r+t-s,s)} \dd{s}. \label{eq:delta-v}
\end{equation} 
\end{lem}
\begin{proof}
Subtracting equations \eqref{eq:equation-sub} and \eqref{eq:tildev} we obtain
\begin{equation}
u_{\Psi} (x,r,t) = \intc{t} \T{t-s} \sbr{-Q u_{\Psi}(\cdot, r+t-s,s) + \Psi(\cdot, r+t-s) v_{\Psi}(\cdot,r+t-s,s) + V q v_{\Psi}^2(\cdot, r+t-s,s)} \dd{s}. \label{eq:Delta-v-1}
\end{equation}
Although we do not know solution of \eqref{eq:equation-sub} we may treat $v_\Psi$ as a known function. It is easy to check that \eqref{eq:delta-v} solves \eqref{eq:Delta-v-1}. Standard application of the Banach contraction principle proves that that it is unique.
\end{proof}
\paragraph{Notation} For now on we fix non-negative $\Phi$ and prove convergences announced in the scheme in Section \ref{sec:scheme}. To make the proof shorter we will consider $\Phi$ of a special form
\begin{equation}
 \Phi(x,s) = \varphi(x)\psi(s), \varphi\in \SD, \psi \in \mathcal{S}(\mathbb{R}),\varphi \geq 0,\psi \geq 0. \label{def:simplification}
\end{equation}
We also denote
\begin{equation}
 \varphi_{T}\left(x\right)=\frac{1}{F_{T}}\varphi\left(x\right), \:\chi(s) = \int_s^1\psi(u) \dd{u},\: \chi_{T}=\chi\left(\frac{t}{T}\right). \label{def:notation-T}
\end{equation}
We write
\begin{equation*}
 \Psi(x,s) = \varphi(x)\chi(s),
\end{equation*}
\begin{equation}
  \Psi_{T}\left(x,s\right)=\frac{1}{F_{T}}\Psi\left(x,\frac{s}{T}\right) = \varphi_T(x)\chi_T(s). \label{def:simpl}
\end{equation}
note that $\Psi$ and $\Psi_T$ are positive functions.
In the sequel, we also write
\begin{equation}
 v_T(x,r,t)=v_{\Psi_T}(x,r,t) \:\text{ and }\: v_T(x) = v_T(x,0,T), \label{eq:simpl1}
\end{equation}
and
\begin{equation*}
 \tilde{v}_T := \tilde{v}_{\Psi_T},\quad u_T := u_{\Psi_T}.
\end{equation*}
It is obvious now that $u_{\Psi}\geq0$ which together with equations \eqref{eq:def-u} and \eqref{sol:tv}  implies
\begin{equation}
  0\leq v_T \leq \tv \leq \frac{C_{\Psi}}{F_T}. \label{ineq:tv>v}
\end{equation}
We will also use the following simple estimation 
\begin{equation}
   u_T (x,r,t) \leq  \frac{C}{F_T^2}.  \label{eq:estimate2}
\end{equation}
Fix $\Psi$ and denote
\begin{equation}
 v(\theta) = v_{\theta \Psi}
\end{equation} 
In the sequel we will need derivatives of $v$ with respect to $\theta$. It is easy to calculate by \eqref{eq:equation-sub} that (we omit arguments and integration variables)
\begin{equation}
 v'(\theta) = \int_{0}^{t}\mathcal{T}_{t-s}\left[\Psi \left(1-v(\theta)\right)-\theta \Psi v'(\theta) - Qv'(\theta) - 2 Vq v(\theta)v'(\theta)  \right].\label{eq:trata1}
\end{equation} 
When $\theta=0$ then
\begin{equation}
 v'(0) = \int_{0}^{t}\mathcal{T}_{t-s}\left[\Psi - Qv'(0)  \right]. \label{eq:vprim}
\end{equation} 
It is easy to notice that it is the same equation as \eqref{eq:tildev} hence $\tilde{v} = v'(0)$ (note that the above calculation is not quite rigorous as one have to justify differentiation under integral in \eqref{eq:trata1}).

\subsection{Laplace transform} \label{sec:laplace}
In this section we calculate the Laplace transform of the space-time variable $\tilde{X}_T$. Let us recall that the initial distribution is given by a Poisson random field with intensity $L\lambda,L>0$ and the immigration is determined by a Poisson random field on $\mathbb{R}_+\times\Rd$ with intensity $H\rbr{\lambda \otimes \lambda},H>0$. We can split the system $N$ into two independent parts
\begin{equation*}
 N_t = N_t^0 + N^{Imm}_t,
\end{equation*}
where $N^0$ consists of particles being in the system at time $t=0$ and their offspring while $N^{Imm}$ is the immigration part with particles which appeared in the system after $t=0$ and their descendants.\\
The first step is calculating the Laplace transform of the space-time variable corresponding to rescaled occupation time process $\tilde{Y}$. It is easy to check that 
\begin{equation}
 \ddp{\tilde{Y}_T}{\Phi} = \frac{T}{F_T} \sbr{\intc{1}\ddp{N_{Ts}}{\Psi(\cdot,s)}\dd{s}} = \sbr{\intc{T}\ddp{N_{s}}{\Psi_T(\cdot,s)}\dd{s}}. \label{eq:laplace-derivation}
\end{equation}
Denote (recall \eqref{def:simpl} for the relation between $\Psi$ and $\Phi$)
\begin{equation}
 K_T(\Phi) = \ev{\exp\rbr{ -\ddp{\tilde{Y}_T }{\Phi}  }  } = \ev{\exp\rbr{-\sbr{\intc{T}\ddp{N_{s}}{\Psi_T(\cdot,s)}\dd{s}}} }.
\end{equation}
We can write
\begin{eqnarray}
K_T(\Phi) =  \ev{\exp{\cbr{-\intc{T}\ddp{N^{0}_{s}}{\Psi_T(\cdot,s)}\dd{s}}}}\ev{\exp{\cbr{-\intc{T}\ddp{N^{Imm}_{s}}{\Psi_T(\cdot,s)}\dd{s}}}}.
\end{eqnarray} 
 Firstly evaluate the term with $N^{Imm}$. Conditioning with respect to $Imm$, using independence of evolution of particles (\textit{branching Markov property}) and equation \eqref{def:v} we obtain
\begin{multline}
 {\ev{\rbr{\left.\exp{\cbr{-\intc{T}\ddp{N^{Imm}_{s}}{\Psi_T(\cdot,s)}\dd{s}}}\right|Imm}}} =\\ \prod_{(t,x) \in \widehat{Imm}} \ev{\exp{\cbr{-\int_t^T\ddp{N^x_{s-t}}{\Psi_T(\cdot,s)}\dd{s}}}} = \prod_{(t,x) \in \widehat{Imm}}\rbr{1-v_{\Psi_T}(x,t,T-t)},
\end{multline}
where $\widehat{Imm}$ is a (random) set such that $\sum_{(t,x) \in\widehat{Imm}} \delta_{(t,x)} =Imm\: a.s.$ and $\delta_{(t,x)}$ corresponds to a particles which immigrate to the system at time $t$ to location $x$. Hence we have
\begin{equation*}
 \ev{\exp{\cbr{-\intc{T}\ddp{N^{Imm}_{s}}{\Psi_T(\cdot,s)}\dd{s}}}} = \ev{\exp\cbr{\ddp{Imm}{\log (1-v_T(\cdot,\star,T-\star)}}}, 
\end{equation*}
where $\cdot$,$\star$ denote integration with respect to space and time, respectively. Taking into account distribution of $Imm$ we obtain
\begin{equation*}
  \ev{\exp{\cbr{-\intc{T}\ddp{N^{Imm}_{s}}{\Psi_T(\cdot,s)}\dd{s}}}} =  \exp\cbr{  -H \intc{T} \intr v_{\Psi_T}(x,T-t,t)  \dd{x} \dd{t}}.
\end{equation*}
The first term is easier and can be treated similarly - we have
\begin{equation}
 \ev{\exp{\cbr{-\intc{T}\ddp{N^{0}_{s}}{\Psi_T(\cdot,s)}\dd{s}}}} = \exp\cbr{  -L  \intr v_{\Psi_T}(x,0,T)  \dd{x}}.
\end{equation} 
Finally we have 
\begin{equation}
 K_T(\Phi) = \exp\cbr{  -H \intc{T} \intr v_{\Psi_T}(x,T-t,t)  \dd{x} \dd{t} - L  \intr v_{\Psi_T}(x,0,T)  \dd{x} }
\end{equation} 
By the properties of the Laplace transform we have (recall also that $v_T'(0) =\tilde{v}_T$ - see \eqref{eq:vprim} and  $v(0)=0$)
\begin{equation*}
 \ev{}\ddp{\tilde{Y}_T}{\Phi} = \left.\frac{d}{\dd{\theta}}\right|_{\theta=0} K_T(\theta \Phi) =
	 { -H \intc{T} \intr \tilde{v}_{\Psi_T}(x,T-t,t)  \dd{x} \dd{t} - L  \intr \tilde{v}_{\Psi_T}(x,0,T)  \dd{x}  }.
\end{equation*}
Now we can calculate the Laplace transform of $\tilde{X}_T$. Using definition of $u_T$ \eqref{eq:def-u}, simple fact that $\tilde{X}_T = \tilde{Y}_T - \ev{\tilde{Y}_T}$ we obtain
\begin{equation}
  L_T(\Phi)=\ev{\exp\cbr{-\ddp{\tilde{X}_T}{\Phi}}} =  \exp\cbr{L\intr u_T(x,0,T)\dd{x}}  \exp\cbr{  H \intc{T} \intr u_T(x,T-t,t)  \dd{x} \dd{t}}. \label{eq:laplace-imm}
\end{equation}
Now the task is to show limit of \eqref{eq:laplace-imm}. Using \eqref{eq:delta-v} one obtains
\begin{eqnarray}
	 \mathbb{E}\exp\left\{ -\left\langle \tilde{X}_{T},  \Phi\right\rangle \right\}  = \exp \cbr{L \rbr{A_1(T) +  A_2(T)} + H \rbr{A_3(T)+A_4(T)}}, \label{eq:split1}
\end{eqnarray}
where
\begin{equation}
 A_1(T) = \intr\: \intc{T}  \T{T-s}^Q\sbr{\Psi_T(\cdot,T-s) v_T(\cdot, T-s,s)}(x) \dd{s} \dd{x}, \label{def:A1}
\end{equation}
\begin{equation}
 A_2(T) = Vq \intr\: \intc{T}  \T{T-s}^Q v_T^2(x, T-s,s) \dd{s} \dd{x},\label{def:A2}
\end{equation}
\begin{equation}
 A_3(T) = \intr\: \intc{T} \intc{t}  \T{t-s}^Q\sbr{\Psi_T(\cdot,T-s) v_T(\cdot, T-s,s)}(x) \dd{s} \dd{t} \dd{x}, \label{def:A3}
\end{equation}
\begin{equation}
 A_4(T) = Vq \intr\: \intc{T} \intc{t}  \T{t-s}^Q v_T^2(x, T-s,s) \dd{s} \dd{t} \dd{x}.\label{def:A4}
\end{equation}
The first part of Theorem \ref{thm:thm1} will be proved once we have established 
\begin{equation}
 A_1(T)\rightarrow 0, A_2(T)\rightarrow 0, \quad \text{ as } T\rightarrow +\infty. \label{lim:a1}
\end{equation}
\begin{equation}
 A_3(T)\rightarrow \intc{1} \chi(1-s)^2 \intr U^Q\sbr{\varphi(\cdot) U^Q \varphi(\cdot) }(x)\dd{x} \dd{s}, \quad \text{ as } T\rightarrow +\infty.\label{lim:a3}
\end{equation}
\begin{equation}
 A_4(T)\rightarrow  2Vq \intc{1} \chi(1-v_1)^2 \int_{0}^{+\infty}  \intr U^Q \sbr{ \T{s}^Q\varphi(\cdot) \T{s}^Q U^Q\varphi(\cdot) }(x) \dd{x}   \dd{s}  \dd{v_1}, \quad \text{ as } T\rightarrow +\infty.\label{lim:a4}
\end{equation}
In the next section we will prove \eqref{lim:a3}, \eqref{lim:a4}. The poofs of \eqref{lim:a1} are simpler and are left to the reader.
\subsection{Tightness} \label{sec:tightness_plus}
Recall that we continue the proof according to the scheme in Section \ref{sec:scheme}. First, we will compute the left hand side of \eqref{eq: laplace-tighness-end}. We adopt the following notation - denote $\Phi_{\theta,T} = \theta \Phi_T$ and  $\Psi_{\theta,T} = \theta \Psi_T= \theta\varphi_T\otimes \chi_T$ related to $\Phi_{\theta,T}$ by equation \eqref{def:simpl}. Additional parameter $\theta$ will indicate that a particular quantity is calculate for $\Phi_{\theta,T}$ or $\Psi_{\theta,T}$. Hence using \eqref{eq:split1} we can write (we additionally assume that $L=0,H=1$, the proof without this assumptions goes exactly the same lines but is longer)
\begin{equation}
 \mathbb{E}\exp\left\{ -\left\langle \tilde{X}_{T},  \theta \Phi\right\rangle \right\} = \exp \cbr{ {A_{3}(\theta,T) +  A_{4}(\theta,T) }}. \label{abc}
\end{equation}
For sake of consistency we denote
\[
v(\theta) := v(\theta)(x, r, t) = v_{\Psi_{\theta, T}}(x, r,t),
\]
Differentiating \eqref{eq:equation-sub} and evaluating at $\theta=0$ yields (we skip arguments and integration variables)
\begin{equation*}
 v(0) = 0,
\end{equation*}
\begin{equation*}
  v'(0) = \intc{u} \T{u-s} \sbr{ \Psi_T  - Q v'(0)},
\end{equation*}
\begin{equation*}
  v''(0) = \intc{u} \T{u-s} \sbr{  - 2 \Psi_T v'(0) - Q v''(\theta)  - 2 Vq v'(0)^2 },
\end{equation*}
\begin{equation*}
  v'''(0) = \intc{u} \T{u-s} \sbr{  - 3 \Psi_T v''(0)  - Q v'''(0)  - 5Vq v''(0) v'(0)  }.
\end{equation*}
These equations can be solved (we skip detailed calculations)
\begin{equation}
v'(0)(x,r,t) = \intc{t} \T{t-s}^Q \Psi_T(x, r+t-s) \dd{s}. \label{kupkaka1}
\end{equation}
\begin{multline}
 v''(0)(x,r,t) = -2 \intc{t} \T{t-s}^Q \left[ \Psi_T(x,r+t-s) v'(0)(x,r+t-s,s)\right. +\\
   \left. Vq v'(0)(x,r+t-s,s)^2\right]  \dd{s},\label{kupkaka2}
\end{multline}
\begin{multline*}
 v'''(0)(x,r,t) = - \intc{t} \T{t-s}^Q \left[  3 \Psi_T(x,r+t-s) v''(0)(x,r+t-s,s)\right. +\\  \left. 5 Vq v''(0)(x,r+t-s,s)v'(0)(x,r+t-s,s) \right] \dd{s}.
\end{multline*}
Differentiating equations \eqref{def:A3} and \eqref{def:A4} and evaluating at $\theta=0$ one gets (in the last expression we skip arguments)
\begin{equation*}
 A_3(0,T) = 0,\quad  A_3'(0,T) = 0,
\end{equation*} 
\begin{equation*}
 A_3^{(i)}(0,T) =  i\: \intr\: \intc{T} \intc{t}  \T{t-s}^Q\sbr{\Psi_T(x,T-s) v^{(i-1)}(0)(\cdot, T-s,s)}(x) \dd{s} \dd{t} \dd{x}, \quad i\geq2.
\end{equation*}
\begin{equation*}
 A_4(0,T) = 0,\quad  A_4'(0,T) = 0,
\end{equation*} 
\begin{equation}
 A_4''(0,T) = 2Vq \intr\: \intc{T} \intc{t}  \T{t-s}^Q v(0)'(x, T-s,s)^2 \dd{s} \dd{t} \dd{x}.
\end{equation}
\begin{equation}
 A_4^{(IV)}(0,T) = Vq \intr\: \intc{T} \intc{t}  \T{t-s}^Q \rbr{v(0)'''v(0)' + (v(0)'')^2} \dd{s} \dd{t} \dd{x}. \label{eq:holahola}
\end{equation}
Now we are ready to differentiate \eqref{abc}
\begin{equation*}
 \left.\frac{d^4}{d\theta^4}\right|_{\theta = 0} \exp \cbr{ {A_{3}(\theta,T) +  A_{4}(\theta,T) }}=
A_{3}^{IV}(0,T) + A_{4}^{IV}(0,T) + 3(A_{3}''(0,T) + A_{4}''(0,T))^2
\end{equation*}
Now in order to show \eqref{eq: laplace-tighness-end} it suffices to prove
\begin{equation}
 A_{3}^{IV}(0,T) \leq c (t-s)^2,\quad A_{4}^{IV}(0,T) \leq c (t-s)^2, \label{eq:mamama1}
\end{equation}
\begin{equation}
 A_{3}'''(0,T) \leq c (t-s),\quad A_{4}'''(0,T) \leq c (t-s).\label{eq:mamama2}
\end{equation}
Example computations will be shown in Section \ref{sec:labt}

\section{Calculations} \label{sec:labx}
\subsection{Calculations - convergence} \label{sec:lab}
\paragraph{Convergence of $A_{3}$} Firstly, we replace  $v$ with $\tilde{v}$. Secondly, we calculate the limit for such expression. In the end we will prove that the change do not affect the limit
\begin{equation*}
 \tilde{A}_{3}(T) = \intr\: \intc{T} \intc{t}  \T{t-s}^Q\sbr{\Psi_T(\cdot,T-s) \tilde{v}_T(\cdot, T-s,s)}(x) \dd{s} \dd{t} \dd{x}.
\end{equation*}
Using equation \eqref{sol:tv} and Fubini's theorem we get
\begin{equation*}
 \tilde{A}_{3}(T) = \intr\: \intc{T} \intc{t} \intc{s} \T{t-s}^Q\sbr{\Psi_T(\cdot,T-s) \T{s-u}^Q \Psi_T(\cdot,T-u) }(x) \dd{u} \dd{s} \dd{t} \dd{x}.
\end{equation*}
Using \eqref{def:simpl} and Fubini's theorem once more one can write
\begin{equation*}
 \tilde{A}_{3}(T) =  \intc{T} \intc{t} \intc{s} \chi_T(T-s) \chi_T(T-u) \intr \T{t-s}^Q\sbr{\varphi_T(\cdot) \T{s-u}^Q \varphi_T(\cdot) }(x)\dd{x} \dd{u} \dd{s} \dd{t}.
\end{equation*}
Changing variables $t\rightarrow Tt$, $s \rightarrow Ts$, $u\rightarrow Ts$ and using \eqref{def:notation-T} we have
\begin{equation*}
 \tilde{A}_{3}(T) =  \frac{T^3}{F_T^2} \intc{1} \intc{t} \intc{s} \chi(1-s) \chi(1-u) \intr \T{T(t-s)}^Q\sbr{\varphi(\cdot) \T{T(s-u)}^Q \varphi(\cdot) }(x)\dd{x} \dd{u} \dd{s} \dd{t}.
\end{equation*}
Changing the order of integration and changing $u\rightarrow s-h$ one obtains
\begin{equation*}
 \tilde{A}_{3}(T) =  \frac{T^3}{F_T^2} \intc{1} \intc{s} \chi(1-s) \chi(1-s+h) \int_{s}^{1} \intr \T{T(t-s)}^Q\sbr{\varphi(\cdot) \T{Th}^Q \varphi(\cdot) }(x)\dd{x} \dd{t} \dd{h} \dd{s}.
\end{equation*}
Finally changing $t\rightarrow t+s$ we obtain
\begin{equation*}
 \tilde{A}_{3}(T) =  \frac{T^3}{F_T^2} \intc{1} \intc{s} \chi(1-s) \chi(1-s+h) \int_{0}^{1-s} \intr \T{Tt}^Q\sbr{\varphi(\cdot) \T{Th}^Q \varphi(\cdot) }(x)\dd{x} \dd{t} \dd{h} \dd{s}.
\end{equation*}
Now it is obvious that 
\begin{equation}
 \tilde{A}_{3}(T) = \tilde{A}_{31}(T) + \tilde{A}_{32}(T),
\end{equation} 
where
\begin{equation*}
 \tilde{A}_{31}(T) =  \frac{T^3}{F_T^2} \intc{1} \chi(1-s)^2 \intc{s}  \int_{0}^{1-s} \intr \T{Tt}^Q\sbr{\varphi(\cdot) \T{Th}^Q \varphi(\cdot) }(x)\dd{x} \dd{t} \dd{h} \dd{s}.
\end{equation*}
\begin{equation*}
 \tilde{A}_{32}(T) =  \frac{T^3}{F_T^2} \intc{1} \chi(1-s) \intc{s} \sbr{\chi(1-s+h)-\chi(1-s)}  \int_{0}^{1-s} \intr \T{Tt}^Q\sbr{\varphi(\cdot) \T{Th}^Q \varphi(\cdot) }(x)\dd{x} \dd{t} \dd{h} \dd{s}.
\end{equation*}
Recall $F_T=T^{1/2}$ and change integration variables $t\rightarrow t/T$ and $h\rightarrow h/T$
\begin{equation*}
 \tilde{A}_{31}(T) =  \intc{1} \chi(1-s)^2 \intc{Ts}  \int_{0}^{T(1-s)} \intr \T{t}^Q\sbr{\varphi(\cdot) \T{h}^Q \varphi(\cdot) }(x)\dd{x} \dd{t} \dd{h} \dd{s}.
\end{equation*}
Fubini's and Lebesgue's monotone convergence theorems imply immediately
\begin{equation}
 \tilde{A}_{31}(T) \rightarrow \intc{1} \chi(1-s)^2 \intr U^Q\sbr{\varphi(\cdot) U^Q \varphi(\cdot) }(x)\dd{x} \dd{s}, \quad \text{ as } T\rightarrow +\infty. \label{eq:conv-1}
\end{equation}
Notice also that by assumption (A5) the integral is finite.
For a $\delta>0$ one can choose $\epsilon>0$ such that $\sup_{h\in (0,\epsilon)}|\chi(1-s+h)-\chi(1-s)|<\delta$ we have
\begin{multline*}
 |\tilde{A}_{32}(T)| \leq  \delta \frac{T^3}{F_T^2} \intc{1}  \intc{\epsilon} \int_{0}^{1-s} \intr \T{Tt}^Q\sbr{\varphi(\cdot) \T{Th}^Q \varphi(\cdot) }(x)\dd{x} \dd{t} \dd{h} \dd{s} + \\\frac{T^3}{F_T^2} \intc{1}  \int_{\epsilon}^s \int_{0}^{1-s} \intr \T{Tt}^Q\sbr{\varphi(\cdot) \T{Th}^Q \varphi(\cdot) }(x)\dd{x} \dd{t} \dd{h} \dd{s} \\.
\end{multline*}
By virtue of convergence \eqref{eq:conv-1} we know that the first integral is finite the second can be written as
\begin{equation*}
 \frac{T^3}{F_T^2} \intc{1}  \inti \mathbf{1}_{(\epsilon,s)}(h) \inti \mathbf{1}_{(0,1-s)}(t) \intr \T{Tt}^Q\sbr{\varphi(\cdot) \T{Th}^Q \varphi(\cdot) }(x)\dd{x} \dd{t} \dd{h} \dd{s} = 
\end{equation*}
Changing integration variables $h\rightarrow h/T$ and $t\rightarrow t/T$ and using $F_T=T^{1/2}$ we have
\begin{equation*}
 \intc{1}  \inti \inti \intr \mathbf{1}_{(T\epsilon,Ts)}(h) \mathbf{1}_{(0,T(1-s))}(t)  \T{t}^Q\sbr{\varphi(\cdot) \T{h}^Q \varphi(\cdot) }(x)\dd{x} \dd{t} \dd{h} \dd{s}
\end{equation*}
The integrand converges point-wise to $0$ and is dominated by $\T{t}^Q\sbr{\varphi(\cdot) \T{h}^Q \varphi(\cdot) }(x)$ (which by virtue of previous argument is integrable) hence by Lebesgue's theorem the integral converges to $0$. $\delta$ can be chosen arbitrarily small, consequently
\begin{equation}
 \tilde{A}_{31}(T) \rightarrow 0.
\end{equation} 
The last step is to estimate the difference $\tilde{A}_3(T) - A_3(T)$. Be definition \eqref{eq:def-u} we have
\begin{equation*}
  \tilde{A}_3(T) - A_3(T) =  \intr\: \intc{T} \intc{t}  \T{t-s}^Q\sbr{\Psi_T(\cdot,T-s) u_T(\cdot, T-s,s)}(x) \dd{s} \dd{t} \dd{x}.
\end{equation*}
We can utilize inequality \eqref{eq:estimate2} hence
\begin{equation*}
  \tilde{A}_3(T) - A_3(T) \leq \frac{c_1}{F_T^2} \intr\: \intc{T} \intc{t}  \T{t-s}^Q {\Psi_T(x,T-s) }(x) \dd{s} \dd{t} \dd{x}.
\end{equation*}
Using notation \eqref{def:simpl} we have 
\begin{equation*}
 \tilde{A}_3(T) - A_3(T) \leq \frac{c_2}{F_T^3} \intr\: \intc{T} \intc{t}  \T{t-s}^Q \varphi(x)\dd{s} \dd{t} \dd{x}.
\end{equation*}
By assumption (A5) it is straightforward to check that this converges to $0$ as $T\rightarrow +\infty$.

\paragraph{Convergence of $A_{4}$} Similarly as for $A_3$ we  replace  $v$ with $\tilde{v}$ and  calculate the limit for such changed expression. In the end of the section we will prove that the change do not affect the limit
\begin{equation}
 \tilde{A}_{4}(T) = Vq \intr\: \intc{T} \intc{t}  \T{t-s}^Q\sbr{ \tilde{v}^2_T(\cdot, T-s,s)}(x) \dd{s} \dd{t} \dd{x}.\label{eq:tilde-a4}
\end{equation}
Firstly we use equation \eqref{sol:tv}
\begin{equation*}
 \tilde{A}_{4}(T) = Vq \intr\: \intc{T} \intc{t} \intc{s} \intc{s}  \T{t-s}^Q\sbr{ \T{s-v_1}^Q\Psi_T(\cdot,T-v_1) \T{s-v_2}^Q\Psi_T(\cdot,T-v_2) }(x) \dd{v_2}\dd{v_1} \dd{s} \dd{t} \dd{x}.
\end{equation*}
Using \eqref{def:simpl} and Fubini's theorem yield
\begin{equation*}
 \tilde{A}_{4}(T) = Vq \intc{T} \intc{t} \intc{s} \intc{s} \chi_T(T-v_1) \chi_T(T-v_2) \intr \T{t-s}^Q\sbr{ \T{s-v_1}^Q\varphi_T(\cdot) \T{s-v_2}^Q \varphi_T(\cdot) }(x)  \dd{x} \dd{v_2}\dd{v_1} \dd{s} \dd{t}.
\end{equation*}
We substitute $t\rightarrow Tt$, $s\rightarrow Ts$, $v_1\rightarrow Tv_1$, $v_2\rightarrow Tv_2$ and use \eqref{def:notation-T}
\begin{equation*}
 \tilde{A}_{4}(T) = Vq \frac{T^4}{F_T^2} \intc{1} \intc{t} \intc{s} \intc{s} \chi(1-v_1) \chi(1-v_2) \intr \T{T(t-s)}^Q\sbr{ \T{T(s-v_1)}^Q\varphi(\cdot) \T{T(s-v_2)}^Q \varphi(\cdot) }(x)  \dd{x} \dd{v_2}\dd{v_1} \dd{s} \dd{t}.
\end{equation*}
Next we change the order of integration and use symmetry
\begin{multline*}
 \tilde{A}_{4}(T) = 2Vq \frac{T^4}{F_T^2} \intc{1} \intc{v_1} \chi(1-v_1) \chi(1-v_2)\\ \int_{v_1}^1 \int_{s}^1  \intr \T{T(t-s)}^Q \sbr{ \T{T(s-v_1)}^Q\varphi(\cdot) \T{T(s-v_2)}^Q \varphi(\cdot) }(x) \dd{x}  \dd{t} \dd{s}  \dd{v_2}\dd{v_1}.
\end{multline*}
Let us now substitute $v_2\rightarrow v_1 - h$, $s\rightarrow s+v_1$, $t\rightarrow t+s$
\begin{multline*}
 \tilde{A}_{4}(T) = 2Vq \frac{T^4}{F_T^2} \intc{1} \intc{v_1} \chi(1-v_1) \chi(1-v_1+h) \\ \int_{0}^{1-v_1} \int_{0}^{1-s-v_1}  \intr \T{Tt}^Q \sbr{ \T{Ts}^Q\varphi(\cdot) \T{T(s+h)}^Q \varphi(\cdot) }(x) \dd{x}  \dd{t} \dd{s}  \dd{h}\dd{v_1}.
\end{multline*}
%
Now it is obvious that
\begin{equation}
 \tilde{A}_{4}(T) = \tilde{A}_{41}(T) + \tilde{A}_{42}(T),
\end{equation} 
where
\begin{equation*}
 \tilde{A}_{41}(T) = 2Vq \frac{T^4}{F_T^2} \intc{1} \chi(1-v_1)^2 \intc{v_1} \int_{0}^{1-v_1} \int_{0}^{1-s-v_1}  \intr \T{Tt}^Q \sbr{ \T{Ts}^Q\varphi(\cdot) \T{T(s+h)}^Q \varphi(\cdot) }(x) \dd{x}  \dd{t} \dd{s}  \dd{h}\dd{v_1},
\end{equation*}
\begin{multline*}
 \tilde{A}_{41}(T) = 2Vq \frac{T^4}{F_T^2} \intc{1} \intc{v_1} \chi(1-v_1) (\chi(1-v_1+h) - \chi(1-v_1+h)) \\\int_{0}^{1-v_1} \int_{0}^{1-s-v_1}  \intr \T{Tt}^Q \sbr{ \T{Ts}^Q\varphi(\cdot) \T{T(s+h)}^Q \varphi(\cdot) }(x) \dd{x}  \dd{t} \dd{s}  \dd{h}\dd{v_1}.
\end{multline*}

Recall that $F_T = T^{1/2}$ and substitute $h\rightarrow h/T$, $t\rightarrow t/T$, $s\rightarrow s/T$

\begin{equation*}
 \tilde{A}_{41}(T) = 2Vq \intc{1} \chi(1-v_1)^2 \intc{Tv_1} \int_{0}^{T(1-v_1)} \int_{0}^{T(1-s-v_1)}  \intr \T{t}^Q \sbr{ \T{s}^Q\varphi(\cdot) \T{(s+h)}^Q \varphi(\cdot) }(x) \dd{x}  \dd{t} \dd{s}  \dd{h}\dd{v_1},
\end{equation*}
Lebesgue's monotone convergence theorem implies 
\begin{equation*}
  \tilde{A}_{41}(T)  \rightarrow 2Vq \intc{1} \chi(1-v_1)^2 \intc{+\infty} \int_{0}^{+\infty} \int_{0}^{+\infty}  \intr \T{t}^Q \sbr{ \T{s}^Q\varphi(\cdot) \T{(s+h)}^Q \varphi(\cdot) }(x) \dd{x}  \dd{t} \dd{s}  \dd{h}\dd{v_1},
\end{equation*}
This can be written a bit shorter with potential notation
\begin{equation*}
  \tilde{A}_{41}(T)  \rightarrow 2Vq \intc{1} \chi(1-v_1)^2 \int_{0}^{+\infty}  \intr U^Q \sbr{ \T{s}^Q\varphi(\cdot) \T{s}^Q U^Q\varphi(\cdot) }(x) \dd{x}   \dd{s}  \dd{v_1},
\end{equation*}
Note that by assumptions (A5) the integral above is finite.\\
Now we fix $\delta>0$ and choose $\epsilon$ such that $\epsilon>0$ such that $\sup_{t\in (0,\epsilon)}|\chi(1-s+h)-\chi(1-s)|<\delta$ we 
\begin{multline*}
 \tilde{A}_{42}(T) = \delta C\frac{T^4}{F_T^2} \intc{1} \intc{\epsilon} \int_{0}^{1-v_1} \int_{0}^{1-s-v_1}  \intr \T{Tt}^Q \sbr{ \T{Ts}^Q\varphi(\cdot) \T{T(s+h)}^Q \varphi(\cdot) }(x) \dd{x}  \dd{t} \dd{s}  \dd{h}\dd{v_1} + \\
\frac{T^4}{F_T^2} \intc{1} \int_{\epsilon}^{v_1} \int_{0}^{1-v_1} \int_{0}^{1-s-v_1}  \intr \T{Tt}^Q \sbr{ \T{Ts}^Q\varphi(\cdot) \T{T(s+h)}^Q \varphi(\cdot) }(x) \dd{x}  \dd{t} \dd{s}  \dd{h}\dd{v_1}
\end{multline*}
It easy to deduce that the first integral is convergent (it is smaller then $\tilde{A}_{41}(T)$ in fact). Let us deal with the second one. It can be written as
\begin{equation*}
 \frac{T^4}{F_T^2} \intc{1} \inti \inti \inti \intr \mathbf{1}_{(\epsilon, v_1)}(h) \mathbf{1}_{(0, 1-v_1)}(s) \mathbf{1}_{(0, 1-s-v_1)}(t)  \T{Tt}^Q \sbr{ \T{Ts}^Q\varphi(\cdot) \T{T(s+h)}^Q \varphi(\cdot) }(x) \dd{x}  \dd{t} \dd{s}  \dd{h}\dd{v_1}
\end{equation*}
Let us substitute $s\rightarrow s/T$, $h\rightarrow h/T$, $t\rightarrow t/T$ and recall that $F_T = T^{1/2}$
\begin{equation*}
 \intc{1} \inti \inti \inti \intr \mathbf{1}_{(T\epsilon, Tv_1)}(h) \mathbf{1}_{(0, T(1-v_1))}(s) \mathbf{1}_{(0, T(1-s-v_1))}(t)  \T{t}^Q \sbr{ \T{s}^Q\varphi(\cdot) \T{s+h}^Q \varphi(\cdot) }(x) \dd{x}  \dd{t} \dd{s}  \dd{h}\dd{v_1}
\end{equation*}
By assumption (A5) the integrand is dominated by integrable function $\T{t}^Q \sbr{ \T{s}^Q\varphi(\cdot) \T{s+h}^Q \varphi(\cdot) }(x)$ hence Lebesgue's theorem implies the convergence to $0$. We can take $\delta$ arbitrarily small hence 
\begin{equation}
 \tilde{A}_{42}(T) \rightarrow 0.
\end{equation} 
We are left with estimation of $\tilde{A}_{4}(T) - A_4(T)$. By equation \eqref{eq:def-u} and inequality \eqref{ineq:tv>v} we have
\begin{equation*}
 \tilde{A}_{4}(T) - A_4(T) \leq 2Vq \intr\: \intc{T} \intc{t}  \T{t-s}^Q\sbr{ u_T\rbr{\cdot,T-s,s} \tilde{v}_T(\cdot, T-s,s)}(x) \dd{s} \dd{t} \dd{x}.
\end{equation*}
Using estimate \eqref{eq:estimate2} and \eqref{sol:tv} we write
\begin{equation*}
 \tilde{A}_{4}(T) - A_4(T) \leq \frac{2Vq}{F_T^2} \intr\: \intc{T} \intc{t}  \T{t-s}^Q\sbr{ \intc{s}\T{s-u}^Q \Psi_T(\cdot,T-s)}(x) \dd{u}\dd{s} \dd{t} \dd{x}.
\end{equation*}
Using \eqref{def:simpl}, after simple calculations, we get
\begin{equation*}
 \tilde{A}_{4}(T) - A_4(T) \leq \frac{2Vq}{F_T^3} \intr\: \intc{T} \intc{t} u \T{u}^Q\varphi(x) \dd{u} \dd{t} \dd{x}.
\end{equation*}
Now, by using d'Hospital rule, it follows easily from (A6) that  
\begin{equation*}
  \tilde{A}_{4}(T) - A_4(T) \rightarrow 0.
\end{equation*}

\subsection{Calculations - tightness} \label{sec:labt}
We are left with proving inequalities \eqref{eq:mamama1} and \eqref{eq:mamama2}. This can be done by evaluating the lfs of the inequalities using equations derived in Section \ref{sec:tightness_plus} and later estimating each of the resulting terms separately. Calculations are quite lengthy, for the sake of brevity in the paper we present only one illustrative example. Namely, consider the terms arising from  the second term of  \eqref{eq:holahola}  
\begin{equation}
D(T) = \intr\: \intc{T} \intc{h}  \T{h-w}^Q \rbr{ (v(0)'')^2(x,T-w,w)} \dd{w} \dd{h} \dd{x}.
\end{equation} 
From \eqref{kupkaka1} and \eqref{kupkaka2} it is easy to notice that $v''(0)\leq c/F_T^2$ hence
\begin{equation}
 D(T) \leq \frac{c}{F_T^2} \intr\: \intc{T} \intc{t}  \T{h-w}^Q { v(0)''(x,T-w,w)} \dd{w} \dd{h} \dd{x}.
\end{equation} 
Now substitute $v'$ with the first term of \eqref{kupkaka2}, we denote this new expression by $D_1$ (the expression resulting from the second term can to be estimated in a similar way)
\begin{equation}
 D_1(T) =  \frac{c}{F_T^2} \intr\: \intc{T} \intc{h}  \T{h-w}^Q \sbr{\intc{w}  \T{w-u}^Q \Psi_T(x,T-u) v(0)(x,T-u,u) \dd{u} } \dd{w} \dd{h} \dd{x}.
\end{equation} 
Finally we use \eqref{kupkaka2} which yields
\begin{equation}
 D_1(T) =  \frac{c}{F_T^2} \intr\: \intc{T} \intc{h}  \T{h-w}^Q \sbr{\intc{w}  \T{w-u}^Q \rbr{\Psi_T(x,T-u) \intc{u} \T{u-v}^Q \Psi_T(x,T-v) \dd{v}}\dd{u} } \dd{w} \dd{h} \dd{x}.
\end{equation} 
Changing the order of integration and using \eqref{def:simpl} get
\begin{equation}
 D_1(T) =  \frac{c}{F_T^4} \intc{T} \intc{h} \intc{w} \intc{u}  \chi_T(T-u) \chi_T(T-v) \intr \T{h-w}^Q \sbr{  \T{w-u}^Q \rbr{\varphi(x) \T{u-v}^Q \varphi(x) } } \dd{v} \dd{u}\dd{w} \dd{h} \dd{x}.
\end{equation} 
Obvious changes of variables gives
\begin{equation*}
 D_1(T) =  \frac{cT^4}{F_T^4} \intc{1} \intc{h} \intc{w} \intc{u}  \chi(1-u) \chi(1-v) \intr \T{T(h-w)}^Q \sbr{ \T{T(w-u)}^Q \rbr{\varphi(x) \T{T(u-v)}^Q \varphi(x) } } \dd{v} \dd{u}\dd{w} \dd{h} \dd{x}.
\end{equation*}
Recall that we are using the scheme presented in Section \ref{sec:scheme} hence inequality \eqref{tmp:chi-n} holds. We apply it to $\chi(1-v)$, use inequality $\T{T(u-v)}^Q \varphi(x)\leq c e^{-TQ(u-v)}$ and integrate with respect to $v$
\begin{equation*}
 D_1(T) \leq  \frac{cT^3}{F_T^4} (t-s)\intc{1} \intc{h} \intc{w}  \chi(1-u) \intr \T{T(h-w)}^Q \sbr{ \T{T(w-u)}^Q {\varphi(x) } } \dd{u} \dd{w} \dd{h} \dd{x}.
\end{equation*}
Changing the order of integration and integrating with respect to $w$ we get
\begin{equation*}
 D_1(T) \leq  \frac{cT^2}{F_T^2} (t-s) \intc{1} \intc{h}  \chi(1-u) T(h-u) \intr \T{T(h-u)}^Q \varphi(x) \dd{u} \dd{h} \dd{x}.
\end{equation*}
Easy calculations yield 
\begin{equation*}
 D_1(T) \leq  c (t-s) \intc{1} \intc{h}  \chi(1-u)  \intr T(h-u) \T{T(h-u)}^Q \varphi(x) \dd{u} \dd{h} \dd{x}.
\end{equation*}
Using assumption (A6) one easily gets 
\begin{equation*}
 D_1(T) \leq  c (t-s)^{2-\epsilon}T^{-\epsilon}.
\end{equation*}

\begin{acknowledgement*}
 The author would like to express his gratitude to Prof. Tomasz Bojdecki for fruitful discussions and remarks during my work on this paper.
\end{acknowledgement*}

\bibliographystyle{abbrv}
\bibliography{branching}

\end{document}